\newcommand{\GS}{\mathcal{GS}}
\newcommand{\HH}{\mathcal{H}}
\newcommand{\mvsps}{MVSPs}

\newcommand{\ff}{\mathbb F}

\newcommand{\fq}{\mathbb{F}_q}
\newcommand{\fqn}{\mathbb{F}_{q^n}}

\newcommand{\F}{\mathbb{F}}
\newcommand{\Q}{Q}

\newcommand{\W}{\mathcal W}
\newcommand{\X}{\mathcal X}
\newcommand{\divv}{\operatorname{div}_{\infty}}

\newcommand{\ftil}{{f}}
\newcommand{\uu}{{u}}
\newcommand{\vv}{{v}}

\newcommand{\gal}{\operatorname{Gal}}

\newcommand{\Conceicao}{Concei{\c c}\~ao}

\newcommand{\tra}[2]{T_{#1}}
\documentclass[a4paper,12pt]{elsarticle}
\usepackage{graphicx}
\usepackage[all]{xy}
\usepackage[english]{babel}
\usepackage{amssymb,amsmath,amsthm}
\usepackage[usenames,dvipsnames]{xcolor}
\usepackage{breqn,cancel}
\usepackage{color}

\usepackage{pstricks,pst-plot}
\usepackage{verbatim}
\usepackage{hyperref}
\hypersetup{colorlinks=true}

\newtheorem{theorem}{Theorem}[section]

\newtheorem{lemma}[theorem]{Lemma}
\newtheorem{proposition}[theorem]{Proposition}

\newtheorem{corollary}[theorem]{Corollary}

\newdefinition{definition}{Definition}
\newdefinition{example}{Example}
\newdefinition{remark}{Remark}[section]

\title{Some elementary abelian p-extensions and a generalization of the Hermitian curve}

\begin{document}


\makeatletter
\def\ps@pprintTitle{%
  \let\@oddhead\@empty
  \let\@evenhead\@empty
  \let\@oddfoot\@empty
  \let\@evenfoot\@oddfoot
}
\makeatother

\begin{frontmatter}

\title{Minimal value set polynomials and a generalization of the Hermitian curve}
\author[herivas]{Herivelto Borges}
\author[rico]{Ricardo \Conceicao}
\address[herivas]{Universidade de S\~ao Paulo, Inst. de Ci\^encias Matem\'aticas e de Computa\c c\~ao, S\~ao Carlos, SP 13560-970, Brazil.}
\address[rico]{Oxford College of Emory University. 100 Hamill Street, Oxford, Georgia 30054.}

\date{July 2010}

\begin{abstract} 
 We use a recent characterization of minimal value set polynomials and $q$-Frobenius nonclassical curves to construct curves that generalize the Hermitian curve. The genus $g$ and the number $N$ of  $\F_q$-rational points of the curves are computed and, for a special family of these curves, we determine the Weierstrass semigroup at the unique point at infinity.  These special curves yield  new examples of Castle curves  and improve on a previous example of Garcia-Stichtenoth of curves with large ratio $N/g$.
\end{abstract}

\end{frontmatter}

\section{ Introduction}
One of the fundamental results in the theory of  curves over a finite field is the Hasse-Weil bound 
$$
|N -(q+1)|\leq 2g\sqrt{q},
$$
which relates the number $N$ of $\fq$-rational points on (smooth, geometrically irreducible projective) curves to its genus $g$ and the size of the finite field. This result has inspired a great amount of work in the area, specially after a construction by Goppa \cite{Goppa_Codes} of linear codes with good parameters from curves with many  rational points, that is, curves with somewhat large ratio $N/g$.
%
%

In addition to coding theory, curves over finite fields with many rational points  have found applications in areas such as  finite geometry \cite{Hirschfel_Projective}, correlation of shift register sequences \cite{lidl_finite_field_83}, and number theory \cite{Moreno_algebraic,Stepanov_Algebraic}, among others.  

Central to the theory of curves over finite fields  and its applications is the notion of a \emph{maximal curve} -- a curve  that attains the Hasse-Weil upper bound. The \emph{Hermitian curve}, defined over $\ff_{q^2}$ by 
$$
y^{q}+y=x^{q+1},
$$
is the classical and most important example of  a curve with this property. Due to its many nice arithmetic and geometric properties, there has been substantial interest in finding curves  over $\fqn$, $n\geq 2$, that  generalizes the Hermitian curve. The  simplest  generalization of the Hermitian curve  is the so-called \emph{norm-trace} curve, which is defined over $\fqn$ by
\begin{equation}\label{norm-trace}
 y^{q^{n-1}}+\cdots+y^q=x^{1+q+q^2+\cdots+q^{n-1}}.
\end{equation}
Note that when $n = 2$ the curve is just the Hermitian curve. Its genus $g$ is $(q^{n-1} -1)(q^{n-1}+q^{n-2}+\cdots+q)/2$ and the number $N$ of $\fqn$-rational points of this curve is $q^{2n-1}+1$. As demonstrated by Geil \cite{Geil_normtrace}, the norm-trace curve can be used to construct  AG codes with good parameters. In addition,  AG codes from the norm-trace curve have been used in different applications to coding  theory, see for instance \cite{MR3015349,MR2980466,MR2836238,MR2287392}.

In 1999, Garcia and Stichtenoth \cite{Garcia_A_class_of_poly_1999424} constructed a generalization of the Hermitian curve for which the ratio $N/g$  is larger than the corresponding one for the norm-trace curve. More specifically, for $n\geq 2$, they show that the curve  defined over $\F_{q^n}$ by 
\begin{equation}\label{GS}
\GS: y^{q^{n-1}}+\cdots+y^q+y=x^{q+1}+x^{1+q^2}+\cdots+x^{q^{n-1}+q^{n-2}}
\end{equation}
has $N=q^{2n-1}+1$  $\F_{q^n}$-rational points and genus $g=(q^{n-1}-1)q^{n-1}/2$. 
Such curve, that  is now commonly known as \emph{the generalized Hermitian curve}, not only  possess a large set of rational points, but  also shares many other interesting properties with the norm-trace curve. For instance, Bulygin \cite{bulygin_generalize_hermitian} (for $p=2$) and Munuera et al.  \cite{munuera_sepulv_Generl_hermitian_code} (for $p>2$) determined the Weierstrass semigroup $H(P_{\infty})$ at the  only point at infinity $P_{\infty}=(0:1:0) \in \GS$. 
 This information is crucial to the construction of algebraic geometric codes (AG codes, for short) with good parameters. Furthermore,  Munuera et al. \cite{Munuera_sepulv_Algebraic_codes_castle}  proves that the norm-trace curve and the generalized Hermitian curve  belong to a class of curves that they call  ``Castle curves" (see Section \ref{sec:weierstrass_semigrp} for definition). Such  curves possess certain arithmetic properties that make them suitable for the construction of AG codes with good parameters.

The main purpose of this paper is to  introduce a class of curves that generalizes not only the Hermitian curve but also the norm-trace curve and the generalized Hermitian curve $\GS$. Similar to the previous generalizations, our alternative generalization of the Hermitian curve are  irreducible curves defined over $\F_{q^n}$ ($n\geq 2$) by
\begin{equation}\label{Frob1}
y^{q^{n-1}}+\cdots+y^q+y =f(x),
\end{equation}
where  $f(x)$ is a polynomial over $\fqn$. In our construction  we consider $f(x)$ to be an element of a large but proper subset of
\begin{equation}\label{W-set}
\mathcal{W}:=\{\text{minimal value set polynomials $F(x)$ over $\fqn$ with $V_F=\F_q$}\},
\end{equation}
where $V_F:=\{F(\alpha):\alpha\in\fqn\}$, and \emph{minimal value set polynomials} (\mvsps) are non-constant polynomials $F$  satisfying
$$
\#V_F=\left\lfloor\dfrac{q^n-1}{\deg F} \right\rfloor+1.
$$  
We denote by $\mathcal{X}$ any  curve defined by \eqref{Frob1} with $f(x)\in \W$. We note that  $\mathcal{X}$ maybe considered as a generalization of the norm-trace curve, the curve  $\GS$ and (consequently)  the Hermitian curve. Indeed, for the norm-trace curve it is easy to see that the polynomial  on the right-hand side of \eqref{norm-trace} is in $\W$. For the generalized Hermitian curve, our recent progress in the characterization of \mvsps\ in \cite{borges_con_charac_mvsp} shows  that the polynomial  on the right-hand side of \eqref{GS} is in $\W$ (for more details, see \cite[Section 6]{Borges}).

Our choice of $f(x)$ is inspired  by  the close connection between MVSPs and $q$-Frobenius nonclassical curves recently discovered by the second author in \cite{Borges}.  Apart from few special cases, the results in \cite[Section 3]{Borges} imply that an irreducible curve
\begin{equation}\label{Frob}
g(y)=f(x),
\end{equation}
defined over $\F_q$, is $q$-Frobenius nonclassical  if and only if $f(x),g(x) \in \F_{q}[x]$ are MVSPs with $V_f=V_g$. Since $y^{q^{n-1}}+\cdots+y^q+y\in \W$,  this result shows that $\X$ is a $q$-Frobenius nonclassical curve. As such, $\X$ is worthy of further investigation since the class of $q$-Frobenius nonclassical curves is  well-regarded as a potential source of curves with many rational points and interesting arithmetic and geometric properties, cf. \cite{Stohr_Voloch,Hefez_Voloch}.


After a preliminary analysis of some curves of type $$y^{q^{n-1}}+\cdots+y^q+y =f(x),$$ a judicious choice of $f(x) \in \mathcal{W}$ yields  a particular  generalization  $\HH$ (Section \ref{sec:many_pnts}) of the Hermitian curve that not only  shares many of  the nice properties satisfied by the Hermitian curve and its previous generalizations, but  also has a ratio $N/g$  at least $q^{\frac{n}{2}-3}$ times bigger than the corresponding ratio for  the curve $\mathcal{GS}$.  Just as in the case of the norm-trace  and the Garcia-Stichtenoth curves, we are able to compute the genus and the number of rational points of $\HH$ (Proposition \ref{prop:curve_many_pnts});   compute the Weierstrass semigroup of its unique point at infinity $Q_\infty$ (Theorem \ref{the:weiertrass_gen}); and show that $\HH$ provides new examples of Castle curves (Corollary \ref{cor:castle_curv}). We should  mention that,  in a companion paper \cite{Borges1}, the present generalization of the Hermitian curve is   investigated from the perspective of Finite Geometry. In \cite{Borges1}, it is shown that some of these curves  yield new complete $(N,d)$-arcs which  are closely related to the  Artin-Schreier curves  studied by Coulter in \cite{Coulter1}.

To finish, we would like to say a few words about the organization of this work and the techniques used. In Section \ref{sec:curve_mvsp}, we discuss some key properties of the underlying polynomials of some of the curves $\mathcal{X}$,  establish  its irreducibility  and   compute its genus. Here we  follow the method  used by Garcia-Stichtenoth for the curve $\mathcal{GS}$ which in turn relies on the theory of  Artin-Schreier  extensions and on the genus formula for elementary abelian $p$-extensions of the rational function field.  However, as we are dealing with a larger family of curves, the process here will be  quite involved. 
In Section \ref{sec:many_pnts}, we specialize the results of Section  \ref{sec:curve_mvsp} to the particular generalization of the Hermitian curve $\HH$ discussed briefly above. With an eye towards future application to AG codes with good parameters, we also compute in ad hoc manner the Weierstrass semigroup $H(Q_{\infty})$ at the point at infinity of $\HH$ and show that $H(Q_{\infty})$ is telescopic (see \cite[Definition 6.1]{Kirfel_Pellikaan_telescopic}). 
As a consequence of this computation, we show that our present work enlarges the class of  Castle curves.

\section{Frobenius nonclassical curves from MVSPs}\label{sec:curve_mvsp}

Let $n\geq 2$ be an integer. As discussed in the introduction, we are interested in the curves $\X$ defined over $\fqn$ by
$$ 
y^{q^{n-1}}+\cdots+y^q+y=f(x),
$$ 
where $f(x)$ is an element of 
$$
\mathcal{W}:=\{\text{\mvsps} \hspace{0.2 cm} F \in \F_{q^n}[x] : V_F=\F_q\}.
$$

As noted before, $\X$ is a $q$-Frobenius nonclassical  curve that generalizes the Hermitian curve. In addition, the following lemma shows that the polynomials in $\W$ satisfy some of the properties that Garcia-Stichtenoth \cite[Section 1]{Garcia_A_class_of_poly_1999424} observed to be key in the construction of curves with many rational points.
\begin{lemma} \label{lem:garcia_prop} Let  $f(x)$ be a polynomial in $\W$. Then:
\begin{itemize}
\item[(a)] The value set of ${f}(x)$ satisfies $V_f\subseteq \fq$.  
\item[(b)] ${f}(x)-\gamma=0$ has a simple root for all but possibly one  $\gamma\in\fq$.
\item[(c)] Let $\gamma\in\fq$. A root of ${f}(x)-\gamma=0$ is not in $\fq$ if and only if its multiplicity is divisible by $p$.
\end{itemize}
\end{lemma}
\begin{proof}
The polynomial ${f}(x)$ is an MVSP with value set $\fq$  by definition. This proves the first item. Items (b) and (c) are just a restatement of Lemma 2.4 in \cite{borges_con_charac_mvsp}.
\end{proof}

This result provides more evidence that the curves defined by polynomials in $\W$ may play an important role in the theory of curves over finite fields. Not incidentally, in \cite{borges_con_charac_mvsp} we extended the characterization of \mvsps\ initiated in \cite{carlitz_pol_minimal_61, mills_pol_minimal_64}, and provided  a  construction of \mvsps\ over $\fqn$ with a given set of values. More importantly, the results  in \cite{borges_con_charac_mvsp} allow us to describe the elements of $\W$ very explicitly. Thus combining the results in \cite{borges_con_charac_mvsp} and \cite{Borges} we obtain a large source of curves that (potentially) have many rational points. Our goal in this section is to consider a large set of polynomials $f(x)$ in $\W$ for which we can prove that the associated $\X$  is absolutely irreducible, and compute its  number of $\fqn$ rational points and the genus of its non-singular projective model. But before reaching such goal, we take a necessary detour to describe the elements of $\W$ more explicitly.

\begin{definition}\label{def:rtuple}
\begin{itemize}
 \item  $t\geq 0$  is an integer.
 \item $\overline{r}=(r_0,r_1,\ldots, r_{t})$ is a strictly increasing $(t+1)$-uple of  integers.
 \item $r_0=0$ and $r_{t+1}=n$.  
\item $T_n(x)=x+x^q+\ldots +x^{q^{n-1}}$ is the \emph{trace polynomial}.
\end{itemize}
\end{definition}
As an extension of \cite[Theorem 4.7]{borges_con_charac_mvsp}, one can show   that any element of $\W$ is an $\fqn$-linear combination of polynomials of the form
\begin{equation}\label{eq:f_definition}
\tilde{f}_{\overline{r}}(x):=T_n\left(x^{1+q^{r_1}+q^{r_2}+\cdots+q^{r_t}}\right) \mod (x^{q^n}-x),
\end{equation} 
see \cite[Theorem 2.5]{Borges} for more details. In fact, we can be more precise: the set   $\mathcal{W} \cup \F_q$  is an $\F_q$-vector space of dimension $2^{n}$  for which an explicit set of generators can be computed (see  \cite[Theorem 4.8]{borges_con_charac_mvsp}).  Notice that  $\W$ can be quite large since $\#\W=q^{2^n}-q$. Therefore, we expect that the study of the curves $\X$ associated to a general polynomial $f(x)\in\W$   to be computationally hard. For that reason in what follows we restrict our attention  only to those polynomials in $\W$  defined by \eqref{eq:f_definition}.  As we will see, even in this particular case the required computations  are quite involved.



%

\subsection{The polynomials}

Although the right-hand side of  \eqref{eq:f_definition} is a convenient way of representing the polynomials that generate $\W$,  it is not well-suited for computations and  for our construction of a generalization of the Hermitian curve. For instance, if we specialize to the case where $\overline{r}=(0,1,2,\ldots,n-1)$ then  $f_{\bar{r}}(x)=nx^{1+q+q^2+\cdots+q^{n-1}}$ is equal to zero if $p\mid n$. In general, it is not clear from \eqref{eq:f_definition} what the degree of $\tilde{f}_{\bar{r}}(x)$ is, or even if $\tilde{f}_{\bar{r}}(x)$ is non-zero. As we prove below, $\tilde{f}_{\bar{r}}(x)=\eta f_{\bar{r}}(x)$ for some non-zero integer $\eta$ and some  monic polynomial $f_{\bar{r}}(x)$. It is not only more convenient to work with the latter polynomials, but additionally there is no loss in doing so. When $\tilde{f}_{\bar{r}}(x)\neq 0$, most of  its interesting properties, including the ones discussed in Lemma \ref{lem:garcia_prop}, are also satisfied by ${f}_{\bar{r}}(x)$. The aim of this section is to rewrite ${f}_{\bar{r}}(x)$ in a form that is more suitable for  computations.

Hereafter, we fix a $(t+1)$-uple $\bar{r}$ as in Definition \ref{def:rtuple} and write $\tilde{f}(x)=\tilde{f}_{\overline{r}}(x)$, for simplicity of notation.

\begin{definition}\label{def1} 
For each $i \in \{0,\ldots,t\}$, we define
\begin{itemize}
\item[(1)]$\delta_{i}=r_{t+1-i}-r_{t-i}$ and $\delta=\min \{\delta_j: j=0,\ldots,t\}$.
\item[(2)]  
$f_{0}(x):=x^{1+q^{r_1}+q^{r_2}+\cdots+q^{r_t}}$ and  let $f_{i}(x):=f_{i-1}(x)^{\delta_{i-1}} \mod (x^{q^n}-x)$, for $1\leq i \leq t$.
\item[(3)] $I_i:=\{j:   f_j(x)=f_i(x)\}$ and $I:=\{\min{I_i}: i=0,\ldots, t\}$.

\item[(4)] $\eta_i=\#I_i$ and $\eta=\eta_0$.

\item[(5)] ${f}(x)={f}_{\bar{r}}(x):=\sum\limits_{e\in I}{T_{\delta_{e}}(f_{e}(x))}$.
\end{itemize}
\end{definition}


\begin{lemma}\label{lem:about_f}
 Following the notation in Definition \ref{def1}, we have that 
\begin{itemize} 
\item[(a)] For $i\in\{1,\ldots,t\}$,
\begin{eqnarray*} 
 \deg f_i(x)& = &1+q^{\delta_{i-1}}+\cdots+q^{\delta_{i-1}+\delta_{i-2}+\cdots+\delta_{0}}(1+q^{r_1}+\cdots+q^{r_{t-i}})\\
& = &1+q^{r_{t-i+2}-r_{t-i+1}}+\cdots+q^{r_{t+1}-r_{t-i+1}}(1+q^{r_1}+\cdots+q^{r_{t-i}}).
\end{eqnarray*}
%
%
%
%
\item[(b)] $\eta=\eta_0=\cdots=\eta_{t}$ and $\eta$ is a divisor of $n$.

\item[(c)]  $\tilde{f}(x)=\eta\, \ftil(x)$.


%
\end{itemize} 
\end{lemma}
\begin{proof}
 The first equality in part (a) follows by induction on $i$. The second is a consequence of the fact that a sum of the form $\sum\limits_{k=i}^{j} \delta_{k \mod (t+1)}$  reduces to a sum of one, two or three terms. We leave the details to the reader.
 
 To prove part (b) first note that from part (a)
\begin{equation}\label{eq:deg_fi}
\deg f_i(x)=q^{n-\delta_i}+ \underbrace{\qquad  \cdots\qquad }_{\text{powers of $q$ $< q^{n-\delta_i}$}},
\end{equation}
 for any $i\in  \{0,\cdots,t\}$. This implies that for $0\leq j\leq \delta_i-1$ we have $q^{j}\deg f_i(x)<q^n$. Therefore $T_{\delta_i}(f_i(x))$ contains  $\delta_i$ terms of degree $<q^n$. By the definition of the $f_i$'s, the polynomial $\sum\limits_{i=0}^{t}{T_{\delta_i}(f_i(x))}$
contains  the monomial $f_{0}(x)^{q^j}\mod (x^{q^n}-x)$, for $0\leq j< \sum\limits_{i=0}^{t} \delta_i$. Since $n=\sum\limits_{i=0}^{t} \delta_i$, it follows that 
\begin{equation}\label{eq:f_sum_tr}
\tilde{ f}(x)= \sum\limits_{i=0}^{t}{T_{\delta_i}(f_i(x))}. 
\end{equation}
If $m$ is a monomial of $\tilde{f}(x)$, then the operation $m^{q^k}\mod (x^{q^n}-x)$ defines an action of $G:=\gal(\fqn|\fq)$ on the set of monomials of $\tilde{f}(x)$ (cf. \cite[Proposition 4.2]{borges_con_charac_mvsp}). Under this action, we identify the  stabilizer of $f_i(x)$ with
$$
G_i:=\{k \mod n:f_i(x)=f_0(x)^{q^k} \mod (x^{q^n}-x)\}.
$$ 
 Now observe that \eqref{eq:f_sum_tr} implies that the only monomials of $f(x)$ with degree coprime to $p$ are the $f_i(x)$'s. As a consequence $G_i\subseteq \{\delta_0,\delta_0+\delta_1,\ldots,\sum_{j=0}^t\delta_j\}$. Therefore, $G_i=I_i$ and $\eta_i=\#G_i$. This shows that  $\eta_i$ divides $n=\#G$.  Since any $f_i(x)$ is, by definition, on the orbit of $f_0(x)$ by the action of $G$, we have that $\eta=\eta_i$ for all $i\in\{1, \ldots,t\}$. This concludes the proof of part (b). 

Part (c) is an easy consequence of \eqref{eq:f_sum_tr} and part (b).
\end{proof}


Observe that the above lemma shows  that  $f(x)$ is not identically zero and  $\deg f(x)=\max\{{\deg T_{\delta_{e}}(f_{e}(x))}:e\in I\}$. The next definition and lemma provide a more explicit way of computing such degree.

\begin{definition}\label{def2}
For $i,j\in\{0,\ldots,t\}$ we  define
\begin{itemize}
 \item $\Delta_{i,j}:=\sum\limits_{\lambda=0}^{j}\delta_{(i-1-\lambda) \mod (t+1)}$.
 \item $S_i:=(\Delta_{i,0},\ldots,\Delta_{i,t})$.
\end{itemize}
\end{definition}

\begin{remark}\label{rem:deg_Tifi}
Notice that with these definitions we can write $\deg f_i(x)=\sum\limits_{j=0}^{t}q^{\Delta_{i,j}-\delta_i}$ and $\deg T_{\delta_i}(f_i(x))=\sum\limits_{j=0}^{t}q^{\Delta_{i,j}-1}$. 
\end{remark}

\begin{corollary}\label{cor:deg_f_delta}
Let $M$ be such that $S_M$ is the largest sequence, in the lexicographic order, amongst the distinct sequences $\{S_e: e \in I\}$. Then $\deg f(x)=\sum\limits_{j=0}^{t}q^{\Delta_{M,j}-1}$. 
\end{corollary}
\begin{proof}
 This is a consequence of Remark \ref{rem:deg_Tifi} and the fact that the entries of the sequence $S_i$ are increasing.
\end{proof}

The next result presents a case where we can compute $\deg f(x)$ exactly.

\begin{corollary}\label{cor:dec_seq_deg}
 If $\delta_0 \leq \cdots \leq \delta_{t}$, then $\deg f(x)=q^{r_1-1}+\cdots+q^{r_t-1}+q^{n-1}$.
\end{corollary}
\begin{proof}
 Notice that $S_i=(\delta_{i-1\mod (t+1)},\ldots,n-\delta_i)$. The hypotheses on $\delta_i$ imply that $S_0$ is the largest sequence in lexicographic order. Then by  Corollary \ref{cor:deg_f_delta}, $\deg f(x)=\sum\limits_{j=0}^t q^{\sum_{i=0}^{j}\delta_{t-i}-1}=\sum\limits_{j=0}^t q^{r_{j+1}-1}$ as desired.  
\end{proof}

We end this section by proving a sequence of lemmata that is used  in the computation of the genus of  curves to be defined in the next section.

\begin{lemma}\label{lemaHelpH}  
If $m \in  \{0,\cdots,t\}$ is such that $\deg f_m(x)$ is maximal, that is 
$$
\deg f_m(x) =\max\{\deg f_i(x) : i=0,\ldots,t \},
$$ then $\delta_m=\delta$. Moreover, if  $\deg f_i(x)$ is not maximal then
$\deg f_m(x)> q^{\delta_i-\delta-1}\deg f_i(x)$.

\end{lemma}

\begin{proof}
From \eqref{eq:deg_fi}, it follows that $\deg f_m(x) \geq \deg f_i(x)$ implies $q^{n-\delta_m}\geq q^{n-\delta_i}$,  which gives $\delta_m\leq \delta_i$ for all $i\in  \{0,\ldots,t\}$.  This proves the first assertion. For the second one, since 
$$
q^{\delta_i-\delta-1}\deg f_i(x)=q^{n-\delta-1}+\underbrace{\qquad  \cdots\qquad }_{\text{powers of $q$ $< q^{n-\delta-1}$}},
$$
we have $q^{\delta_i-\delta-1}\deg f_i(x)<\deg f_m(x)$, which finishes the proof.
  \end{proof}

\begin{lemma}\label{Lema do f}
There exist polynomials $u(x)$ and $v(x)$ such that 
\begin{equation}
 f(x)=T_\delta(u(x))+v(x)^{q^\delta},
\end{equation}
$\deg u(x)\equiv 1\mod p $, and $\deg v(x)<\deg u(x)$. Moreover,  $\deg u(x)=\deg f_{M}(x)$ where $M$ is such that $S_M$ is the largest sequence, amongst the distinct sequences $\{S_e: e \in I\}$.
\end{lemma}
\begin{proof}
%

Using the minimality of $\delta$ we are able to write 
$$
f(x)= \sum\limits_{e\in I}T_{\delta_e}(f_e(x))={T_{\delta}\left( \sum\limits_{e\in I} f_e(x)\right)}+ \left(\sum\limits_{e\in I} T_{\delta_e-\delta}(f_i(x))\right)^{q^{\delta}},
$$
with the assumption that $T_0\equiv0$. We set $u(x):=\sum_{e\in I} f_e(x)$ and $v(x):=\sum_{e\in I}T_{\delta_e-\delta}(f_e(x))$. Note that $u(x)$ is not identically zero, since  $f_{e_1}(x)\neq f_{e_2}(x)$ for distinct $e_1,e_2\in I$. Considering $m$ as given in Lemma \ref{lemaHelpH}, then $\deg u(x)=\deg f_m(x)$. Observe that this also shows that $ \deg u(x)\equiv 1 \mod p$. 
As for the polynomial $v(x)$, one can see that
$$
\deg v(x)=\max\{q^{\delta_e-\delta-1} \deg f_e(x): \delta_e\neq \delta, e\in I\}.
$$ Now, since  $\delta_e\neq \delta$, it follows from Lemma \ref{lemaHelpH} that $\deg f_e(x)$ is not maximal, and  then that $q^{\delta_e-\delta-1} \deg f_e(x) \leq \deg f_{m}(x)$. Hence  $\deg v(x)<\deg u(x)$.
\end{proof}
\subsection{The curves}

Similar to the case of the curves $\GS$ defined by \eqref{GS}, our  generalization of the Hermitian curve is given by certain elementary abelian $p$-extensions of a rational function field $\bar{\ff}_{q}(x)$. Namely, we consider the curves  defined over $\fqn$ by
\begin{equation}\label{eq:curve_gen_hermitian}
y^{q^{n-1}}+\ldots +y^{q}+y=\ftil(x),
\end{equation}
where $f(x):=f_{\bar{r}}(x)$ is given by Definition \ref{def1}(5). The goal of this section is to compute the genus of a non-singular projective model $\mathcal{F}$ of  such curve and its number of rational points. 

The following two lemmas are used in the proof of the absolute irreducibility of the  curve given by \eqref{eq:curve_gen_hermitian}.
\begin{lemma}\label{lem:trace_coprime}
Let $\alpha \in \bar{\ff}_q$ and  suppose $\tra{m}{}(\alpha)=\tra{n}{}(\alpha)=0$. If $\gcd(m,n)=1$ then $\alpha=0$.
\end{lemma}
\begin{proof}
Note that $\tra{m}{}(\alpha)=\tra{n}{}(\alpha)=0$ implies  $\alpha\in\ff_{q^m}\cap\ff_{q^n}=\fq$, by the coprimality condition. But then $0=\tra{m}{}(\alpha)=m\alpha$ and $0=\tra{n}{}(\alpha)=n\alpha$. Thus $\alpha=0$, otherwise we arrive at the contradiction $p\mid \gcd(m,n)=1$.
\end{proof}

\begin{lemma}\label{lm:NewPoly}
Let $n$ and $m$ be non-negative integers, with $n\geq m$. The polynomial in $\fq[x,y]$ defined by 
\begin{equation*}
S=S_{m,n}(x,y):= \begin{cases} 
0, & \text{ if }\quad m=0\\
\sum\limits_{i=0}^{m-1}y^{q^{n-1-i}}T_{m-i}(x), & \text{ if }\quad m>0\\
\end{cases} 
  \end{equation*}
satisfies $S^q-S=y^{q^{n}}T_{m+1}(x)-xT_{m+1}(y^{q^{n-m}}).$
\end{lemma}
\begin{proof}
The case $m=0$ is trivial, thus we  assume that $m>0$. 

Since 
$S= y^{q^{n-1}}T_{m}(x)+ y^{q^{n-2}}T_{m-1}(x)+\cdots +y^{q^{n-m}}T_{1}(x),$ we have 
\begin{eqnarray*} 
S^q & = & y^{q^{n}}(T_{m+1}(x)-x)+y^{q^{n-1}}(T_{m}(x)-x)+ \cdots +y^{q^{n-m+1}}(T_{2}(x)-x)\\
& = & y^{q^{n}}T_{m+1}(x)+\cdots +y^{q^{n-m+1}}T_{2}(x)-xT_{m+1}(y^{q^{n-m}})+y^{q^{n-m}}x\\
& = & y^{q^{n}}T_{m+1}(x)+S-y^{q^{n-m}}T_{1}(x)-xT_{m+1}(y^{q^{n-m}})+y^{q^{n-m}}x\\
& = & y^{q^{n}}T_{m+1}(x)+S-xT_{m+1}(y^{q^{n-m}}),
\end{eqnarray*}
which gives $S^q-S$ as claimed.
\end{proof}


We are ready to prove the main result of this section.
\begin{theorem}\label{the:herm_gen} Let $n\geq 2$ be an integer, and $\delta$ be given as in Definition \ref{def1}. Let $\uu(x)$ be the polynomial over $\fqn$ given as in Lemma \ref{Lema do f}. If $\gcd(\delta,n)=1$ then the curve defined by \eqref{eq:curve_gen_hermitian}
is  irreducible and its non-singular projective model $\mathcal{F}$ has genus 
$$
g(\mathcal{F})=\dfrac{(q^{n-1}-1)(\deg \uu(x)-1)}{2},
$$ 
and  
$$
\#\mathcal{F}(\F_{q^n})=q^{2n-1}+1.
$$
\end{theorem}

\begin{proof} 

Our proof relies on few facts of the Artin-Schreier theory. For more details, we refer to the introduction of \cite[Section 1]{Garcia_elementary_abelian_91} and the reference therein. 

Denote by $F=K(x)$ the rational function field over $K=\bar{\ff}_q$, and consider the field extension $E/F$, where $E=K(x,y)$ and $x$ and $y$ satisfy \eqref{eq:curve_gen_hermitian}.   We will prove that \eqref{eq:curve_gen_hermitian} is absolutely irreducible, by showing that $[E:F]=q^{n-1}$. For this matter, define $\wp:X \mapsto X^p-X$ to be the Artin-Schreier operator on $K(x)$ and 
$$
A:=\left\{-T_{\delta}(\alpha^{q^{n-\delta+1}})\uu(x)-\alpha^{q^{n-\delta}} \vv(x): T_n(\alpha)=0\right\} \subseteq K(x),
$$ 
where the polynomials $\uu(x)$ and $\vv(x)$ are such that $\ftil(x)=T_{\delta}(\uu(x))+\vv(x)^{q^{\delta}}$,  $\deg \uu(x)\equiv 1 \mod p$ and  $\deg \vv(x) < \deg \uu(x)$ (cf. Lemma \ref{Lema do f}).
It is clear that $A \subseteq K[x]$ is  an additive subgroup, and that $\gcd(\delta,n)=1$ and Lemma \ref{lem:trace_coprime} imply $|A|=q^{n-1}$.  We claim that 
$$
A\cap \wp (K(x))=\{0\}.
$$ 
In fact, suppose that  $ w:=-T_{\delta}(\alpha^{q^{n-\delta+1}})\uu(x)-\alpha^{q^{n-\delta}}v(x)$ is such that $w\in A\cap \wp (K(x))$. That is, the polynomial $X^p-X-w \in K[x][X]$ has a root in $K(x)$. Since UFDs are integrally closed, such a root must be a polynomial and then  $p\mid \deg w$.
On the other hand, since $\deg \uu(x)\equiv 1 \mod p$ and  $\deg \vv(x) < \deg \uu(x)$, it follows that $T_{\delta}(\alpha^{q^{n-\delta+1}})=0$, and then $T_{\delta}(\alpha)=0$. But since $T_{n}(\alpha)=0$ and $\gcd(\delta,n)=1$, Lemma \ref{lem:trace_coprime} implies $\alpha=0$  and consequently $w=0$. 

 From  Artin-Schreier theory,   $F(\wp^{-1}(A))$ is an abelian extension of $F$ of degree $|A|=q^{n-1}$. Therefore, $[E:F]=q^{n-1}$ and \eqref{eq:curve_gen_hermitian} is absolutely irreducible, as desired, if we can show that $F(\wp^{-1}(A))\subset E$.

In order to prove that $F(\wp^{-1}(A))\subset E$, we chose an element $\alpha\in K^*$ such that  $\tra{n}{}(\alpha)=0$, and  let  $S_{m,n}(x,y)$ be the polynomial given in Lemma \ref{lm:NewPoly}.  Consider   the element $z\in E$ given by
$$
z=S_{n-1,n-1}(\alpha,y)+S_{\delta-1,n}(\uu(x),\alpha)+T_{\delta}(\alpha^{q^{n-\delta}}\vv(x)).
$$
Lemma \ref{lm:NewPoly} gives
{\small\begin{eqnarray*} 
z^q-z & = & -\alpha T_{n}(y)+\alpha^{q^{n}}T_{\delta}(\uu(x))-\uu(x)T_{\delta}(\alpha^{q^{n-\delta+1}})+(\alpha^{q^{n-\delta}} \vv(x))^{q^{\delta}}-\alpha^{q^{n-\delta}} \vv(x)\\
& =&-\alpha f(x) +\alpha T_{\delta}(\uu(x)) +\alpha \vv(x)^{q^{\delta}}-T_{\delta}(\alpha^{q^{n-\delta+1}})\uu(x)-\alpha^{q^{n-\delta}} \vv(x) \\
& =&-T_{\delta}(\alpha^{q^{n-\delta+1}})\uu(x)-\alpha^{q^{n-\delta}} \vv(x).
\end{eqnarray*}}
Write $q=p^e$. If we let $z_t:=z^{p^{e-1}}+\cdots+z^p+z$ then $z_t\in E$ and
$$
z_t^p-z_t=z^q-z=-T_{\delta}(\alpha^{q^{n-\delta+1}})\uu(x)-\alpha^{q^{n-\delta}} \vv(x).
$$
Therefore, $E$ contains the splitting fields of  $X^p-X-a$, with $a \in A$ and consequently 
$F(\wp^{-1}(A))\subset E$ as desired.

We now use the fact that $E$ is an elementary abelian $p$-extension to compute its genus. It is  known from  Artin-Schreier theory that $E$ contains $\nu=(q^n-1)/(p-1)$  distinct degree-$p$  sub-extensions  $E_i \subseteq E$ given  by the splitting fields of $X^p-X=a$, with $a \in A^*$.
Note that since $p\nmid \deg \uu(x)$, the genus of each such extension is $g(E_i)=(p-1)(\deg \uu(x)-1)/2$. By \cite[Theorem 2.1]{Garcia_elementary_abelian_91}, the genus of the curve $\mathcal{F}$ can be computed by the formula
$$
g(\mathcal{F})=g(E)=\sum\limits_{i=1}^{\nu}  g(E_i) -\dfrac{p}{p-1}g(F).
$$
Note that in our case $g(F)=g(K(x))=0$. Therefore 
$$
g(\mathcal{F})=\frac{q^{n-1}-1}{p-1}\cdot \frac{(\deg \uu(x)-1)(p-1)}{2}=\frac{(\deg \uu(x)-1)(q^{n-1}-1)}{2}.
$$


To compute the number of $\fqn$-rational points of $\mathcal{F}$, we first note that the pole of $x$ is totally ramified in $E/F$ (see e.g. \cite[Theorem 3.3]{Doelalikar}). That gives us one rational point of $\mathcal{F}$, namely $Q=(0:1:0)$. Lemma \ref{lem:garcia_prop}(1) implies that  for any $\alpha\in\fqn$ we have  $\beta:=f(\alpha)\in\fq$. There are $q^{n-1}$ elements $\gamma\in\fqn$ satisfying
$$
\gamma^{q^{n-1}}+\cdots + \gamma^{q}+\gamma=\beta=\ftil(\alpha).
$$
This shows that  there are $q^n\cdot q^{n-1}+1=q^{2n-1}+1$ $\fqn$-rational points on $\mathcal{F}$.
\end{proof}

\begin{corollary}\label{coro do g} If $\delta_0 \leq \cdots \leq \delta_{t}$ then $\deg f(x)=q^{r_1-1}+\cdots+q^{r_t-1}+q^{n-1}$. Moreover,  the curve $\mathcal{F}$ has genus 
$$
g=\dfrac{1}{2} (q^{n-1}-1)\sum\limits_{i=0}^{t-1}q^{\sum_{j=0}^{i}\delta_{t-j}}.
$$
\end{corollary}

\begin{proof} See Corollary \ref{cor:dec_seq_deg} for the computation of $\deg f(x)$. From Theorem \ref{the:herm_gen}, the formula for the genus  will follow if we can show that $\deg u(x)=\sum\limits_{i=0}^{t-1}q^{\sum_{j=0}^{i}\delta_{t-j}}$. But the latter equality  is just an immediate consequence of  Lemma \ref{Lema do f} in the case $\delta_0 \leq \cdots \leq \delta_{t}$. This finishes the proof.
\end{proof}


\section{An alternative generalization of the Hermitian curve and its properties}\label{sec:many_pnts}

\subsection{Curves with many rational points}
As discussed in the introduction and the references therein, the construction of curves with large ratio $N/g$ is usually very challenging and have applications to  coding theory. In the previous section we have proved that our generalization $\mathcal{F}$ of the Hermitian  curve has $N=q^{2n-1}+1$ $\F_{q^n}$-rational points. Their genus $g$, however, is directly proportional to $\deg u(x)=\mathcal{O}(q^{n-\delta})$, where $\delta=\min\{\delta_i: i=0,\ldots,t\}$. Therefore the largest ratio $N/g$ will be achieved when $\delta$ is large compared to $n$. Clearly the best case scenario is when $ \delta \approx n/2$, and that can only happen when $t=1$. Such conditions lead us to highlight  the following particular case of the curves discussed in the previous section.

Let $n \geq 2$ be an integer and define $r=r(n)\geq n/2$ as being 
the smallest   integer such that $\gcd(n,r)=1$, i.e.,
   \begin{equation}\label{r1}
r= \begin{cases} 
1, & \text{ if }\quad n=2\\
n/2+1, & \text{ if }\quad n\equiv0\pmod4\\
n/2+2, & \text{ if }\quad n\geq 6,\, n\equiv2\pmod4\\
(n+1)/2, & \text{ if }\quad n \text{ is odd}\, .
\end{cases} 
  \end{equation}

 Let  $\mathcal{H}$ be the projective plane curve defined by the affine equation
\begin{equation}\label{eq:curve_many_pnts}
y^{q^{n-1}}+\ldots +y^{q}+y=f_{\bar{r}}(x),
\end{equation}
where $\bar{r}=(0,r)$ and $f_{\bar{r}}(x)$ is given as in Definition \ref{def1}(5). 
 
\begin{proposition}\label{prop:curve_many_pnts}

\begin{enumerate} If $\HH$ is as defined above, then:
 \item  $\mathcal{H}$ has 
degree $q^{n-1}+q^{r-1}$, genus $q^r(q^{n-1}-1)/2$ 
and its number of $\F_{q^n}$-rational points is $q^{2n-1}+1$.
\item  $\HH$ has 
just one point  at infinity, namely  $Q=(0:1:0)$, which is 
also its only singular point whenever $n\geq 3$.
\item  Furthermore, the 
curve is ${q^n}$-Frobenius nonclassical. 
\end{enumerate}

   \end{proposition}

\begin{proof}
We follow the notation in Section \ref{sec:curve_mvsp}. Since $\bar{r}=(0,r)$ and $r>n/2$ we have $\delta_0=n-r<r=\delta_1$. That gives us the two distinct sequences  $S_0=(\delta_1,\delta_1+\delta_0)$ and $S_1=(\delta_0,\delta_0+\delta_1)$ with $S_1<S_0$ in the lexicographical ordering. Also note that $\gcd(r,n)=1$
implies that $n-r=\delta=\min\{\delta_0,\delta_1\}$ is coprime to $n$. 

Part (1) is  a direct consequence of Theorem \ref{the:herm_gen} and Corolary \ref{coro do g}. Part (2) follows from the proof of Theorem \ref{the:herm_gen} and the Jacobian criterion, while part (3) follows from \cite[Corollary 3.8]{Borges}.
\end{proof}

%

The generalized Hermitian curve  $\GS$  has $N=q^{2n-1}+1$  $\F_{q^n}$-rational points and genus $g=(q^{n-1}-1)q^{n-1}/2$. Note that our curve $\mathcal{H}$ has the same number of  $\F_{q^n}$-rational points, but a 
genus $(q^{n-1}-1)q^{r}/2$, which is considerably smaller than $(q^{n-1}-1)q^{n-1}/2$ (since $r \approx n/2$). For example, for  $q=2$ and $n=5$  the curve $\GS$ has $513$ $\F_{32}$-rational points and genus $g(\GS)=120$, while the curve $\mathcal{H}$ has $513$ $\F_{32}$-rational points, but  genus $g(\mathcal{H})=60$. We should mention that,  by the Osterl\'e bound (see e.g. \cite{Osterle}),  a curve over $\F_{32}$ and with $512$ rational points has genus $g\geq 57$.

\subsection{Castle curves and the Weierstrass semigroup at the point at infinity of $\HH$ }\label{sec:weierstrass_semigrp}

Let $Q$ be an $\fqn$-rational point in a curve $X$ over $\fqn$ and write its  Weierstrass semigroup  as $H(Q)=\{m_1=0<m_2<\ldots\}$. The curve $X$ is called a \emph{Castle curve} if $H(Q)$ is symmetric and $\#X(\fqn) = q^n m_2 + 1$. As discussed in the introduction of \cite{Munuera_sepulv_Algebraic_codes_castle}, AG codes are usually difficult to handle  if the geometry of the associated curve is not well understood. Because Castle curves possess some geometric properties that provide a good handling of the parameters of the  associated code, they are well suited for constructing AG codes. In fact, many of the examples of curves yielding AG codes with good parameters are Castle curves. As proved by Bulygin \cite{bulygin_generalize_hermitian}, in the case $q=2$, and  Munuera, Sep\'ulveda, and Torres \cite{munuera_sepulv_Generl_hermitian_code} in the general case, the curve $\GS$ yields linear codes with new records on the parameters. Not incidentally, \cite{Munuera_sepulv_Algebraic_codes_castle} proves that $\GS$ is a Castle curve.
As is usual in this kind of application, the starting point of such results was the computation of the Weierstrass semigroup at $P=(0:1:0) \in \GS$. 
\begin{theorem}[{\cite[Proposition 2.2]{munuera_sepulv_Generl_hermitian_code}}]
The Weierstrass semigroup at $P$ is
$$
H(P)=\langle q^{n-1},q^{n-1}+q^{n-2},q^n+1\rangle.
$$
In particular it is symmetric.
\end{theorem}

%


Given that the norm-trace and the generalized Hermitian curves are Castle curves, it is tempting to  suspect that many of the curves $\X$ (as defined in Section \ref{sec:curve_mvsp}) are also Castle Curves, and may be used for the construction of AG codes with good parameters. The aim of this section is to use the curve   $\mathcal{H}$ to provide more evidence to support this  idea. For this matter, we compute
the Weierstrass semigroup at the unique point at infinity of $ \HH$.

Following the notation in the proof of Theorem \ref{the:herm_gen}, we let $E=K(x,y)$ be the function field of the curve defined by \eqref{eq:curve_many_pnts}, and $\Q=(0:1:0) $ be the only pole of the function $x \in E$. 

\begin{theorem}\label{the:weiertrass_gen}

 Suppose $n\geq 3$, and let  $H(Q)$ be  the Weierstrass semigroup at $\Q$.  Then
$$
H(\Q)=\left\langle q^{n-1},q^{n-1}+q^{r-1},q^n+q^{n-r},q^{2r-1}+q^{n-r-1}, q^{2r}-q^n+q^r+1\right\rangle.
$$
Moreover,  $H(\Q)$ is a telescopic semigroup and, in particular, symmetric.  

\begin{corollary}\label{cor:castle_curv}
 $\HH$ is a Castle curve.
\end{corollary}
\begin{proof}
It follows directly from the fact that 
$$
H(Q)=\left\langle q^{n-1},q^{n-1}+q^{r-1},q^{2r-1}+q^{n-r-1},q^{2r}-q^n+q^r+1\right\rangle
$$ is symmetric, $m_2=q^{n-1}$ and $\mathcal{F}(\fqn)=q^{2n-1}+1=q^nm_2+1$.
\end{proof}

%
%

%
%

\end{theorem}

The reminder of this section is used to prove Theorem \ref{the:weiertrass_gen}  via  a collection of partial results. The following remark  is key in most of our computations.

\begin{remark}\label{rem.Id}
It follows from \eqref{eq:curve_many_pnts} that
\begin{equation}\label{eq:identity_yqn}
 y^{q^n}=y-x^{q^{n-r}+1}-x^{q^r+1}+x^{q^{n-r}+q^n}+x^{q^n+q^r}.
\end{equation}

\end{remark}

\begin{lemma}
Let $s=x^{q^{2r-n}-1}\, y-x^{1+q^r}+y^{q^r}-x^{q^{2r-n}+q^r}$ be a function in $E$. Then:
\begin{itemize}
 \item $\divv(x)=q^{n-1}\Q$;
 \item  $\divv(y)=(q^{n-1}+q^{r-1})\Q$ ; and
 \item  $\divv(s)=(q^{2r-1}+q^{n-r-1})\Q$.
\end{itemize}
\end{lemma}

\begin{proof}
 See proof of Theorem \ref{the:herm_gen}, for the fact that $\divv(x)=q^{n-1}\Q$. From \eqref{eq:curve_many_pnts}, it follows that $\Q$ is the only pole of $y$. To compute its order, let $v_Q$ be the valuation at $\Q$ and notice that
 $$
 q^{n-1}v_\Q(y)=v_Q(f(x))=(q^{n-1}+q^{r-1})v_\Q(x).
 $$
 This implies $\divv(y)=(q^{n-1}+q^{r-1})\Q$. For the function $s$, note that \eqref{eq:identity_yqn} gives
\begin{equation*}
\begin{split}
 s^{q^n}=x^{q^{2r}-q^{n}}(y-x^{q^r+1}-x^{q^{n-r}+1}+x^{q^n+q^r}+x^{q^{n-r}+q^n})-x^{q^{n+r}+q^n}\\+(y-x^{q^r+1}-x^{q^{n-r}+1}+x^{q^n+q^r}+x^{q^{n-r}+q^n})^{q^r}-x^{q^{2r}+q^{n+r}}.
\end{split} 
\end{equation*}
 Therefore,
\begin{equation}\label{eq:sqn}
  s^{q^n}=x^{q^{2r}-q^{n}}(y-x^{q^r+1}-x^{q^{n-r}+1}+x^{q^n+q^{n-r}})+(y-x^{q^{n-r}+1})^{q^r}.
\end{equation}
 Since we know the order of each term on the right-hand side of the above equation, the triangle inequality gives us
  $$
  v_\Q(s^{q^n})=v_\Q(x^{q^{2r}+q^{n-r}})=-q^{n-1}({q^{2r}+q^{n-r}}).
  $$ 
  Thus
$$
\divv(s)=(q^{2r-1}+q^{n-r-1})Q.
$$

 \end{proof}

\begin{lemma}
For  $n\geq 3$, define the following functions in $E$:
\begin{itemize}
\item $w_0=y+y^{q^{r}}-x^{1+q^r}-x^{q^{2r-n}+q^r}$
\item  
$w_1= \begin{cases} 
y^{q^{n-r}}- x^{1+q^{n-r}}+\sum\limits_{i=0}^{\frac{2n-3r-1}{2r-n}}w_0^{q^{1+(2r-n)i}}, & \text{ if  $r\not \equiv 3 \mod 4$}\quad \\
y^{q^{n-r+1}}- x^{q+q^{n-r+1}}+\sum\limits_{i=0}^{\frac{2n-3r+1}{2r-n}}w_0^{q^{(2r-n)i}}, & \text{ if  $r\not \equiv 1 \mod 4$}\quad 
\end{cases} $
\item  $w_2:=x^{q^{2r-n+1}-q}w_1-s^q+x^{q^{2r-n+1}-q^{2r-n}-q+1}s$.
\end{itemize} 
Then $\divv(w_1)=(q^{n}+q^{n-r})Q$ and $\divv(w_2)=(q^{2r}-q^{n}+q^{r}+1)Q$.
\end{lemma}

\begin{proof}
\begin{enumerate}   
\item [(i)] 
 We will  prove  the case $r\not \equiv 3  \mod 4$; the other one  is analogous. To find the order of  $w_1^{q^n}$, first note that from \eqref{eq:identity_yqn}
  \begin{itemize}
\item $(y^{q^{n-r}}- x^{1+q^{n-r}})^{q^n}=y^{q^{n-r}}-x^{q^{n-r}+q^{2n-2r}}-x^{q^{n-r}+q^{n}}+x^{q^{2n-r}+q^{2n-2r}}$, and 
\item $w_0^{q^n}=y+y^{q^r}-x^{1+q^{n-r}}-x^{1+q^r}+x^{q^{n-r}+q^n}-x^{q^r+q^{2r}}$
\end{itemize} 
The assumption on $r$ implies that $t:=(n-r-1)/(2r-n)$ is an integer. Thus we have
\begin{equation*}
\begin{split}
w_1^{q^n}=y^{q^{n-r}}-x^{q^{n-r}+q^{2n-2r}}-x^{q^{n-r}+q^{n}}+x^{q^{2n-r}+q^{2n-2r}}\\
+ \sum\limits_{i=0}^{t-1}(y+y^{q^r}-x^{1+q^r})^{q^{1+(2r-n)i}}-  \sum\limits_{i=0}^{t-1}(x^{1+q^{n-r}})^{q^{1+(2r-n)i}}\\
+\sum\limits_{i=0}^{t-1}(x^{{q^{n-r}+q^n}})^{q^{1+(2r-n)i}}-\sum\limits_{i=0}^{t-1}({x^{q^r+q^{2r}}})^{q^{1+(2r-n)i}}.
\end{split}
\end{equation*}  
 It can be easily checked that
 \begin{itemize}
\item$\sum\limits_{i=0}^{t-1}(x^{1+q^{n-r}})^{q^{1+(2r-n)i}}=x^{q+q^{n-r+1}}-x^{q^{n-r}+q^{2n-2r}}+\sum\limits_{i=0}^{t-1}(x^{q^{2r-n}+q^r})^{q^{1+(2r-n)i}}$\\
 and 
 \item  $\sum\limits_{i=0}^{t-1}(x^{{q^{n-r}+q^n}})^{q^{1+(2r-n)i}}-\sum\limits_{i=0}^{t-1}(x^{{q^{r}+q^{2r}}})^{q^{1+(2r-n)i}}=x^{{q^{n-r+1}+q^{n+1}}}-x^{{q^{2n-2r}+q^{2n-r}}}$
 \end{itemize} 
   Therefore 
   \begin{equation}\label{f4qn}
   w_1^{q^n}=w_1-x^{q+q^{n-r+1}}+x^{{q^{n-r+1}+q^{n+1}}}+x^{1+q^{n-r}}-x^{q^{n-r}+q^{n}},
   \end{equation}
and since  $v_Q(w_1^{q^n}-w_1)=v_Q(-x^{q+q^{n-r+1}}+x^{{q^{n-r+1}+q^{n+1}}}+x^{1+q^{n-r}}-x^{q^{n-r}+q^{n}})$,
triangle inequality gives $q^nv_Q(w_1)=v_Q(x^{{q^{n-r+1}+q^{n+1}}})$. Hence $v_Q(w_1)=-(q^{n-r}+q^n)$.

\item [(ii)] From $w_2^{q^n}=x^{q^{2r+1}-q^n}w_1^{q^n}-s^{q^{n+1}}+x^{q^{2r+1}-q^{2r}-q^{n+1}+q^n}s^{q^n},$
 and equations \eqref{eq:sqn} and  \eqref{f4qn}, we obtain
 
 \begin{equation*}
\begin{split}
 w_2^{q^n}=x^{q^{2r+1}-q^{n+1}}w_1+x^{q^{2r+1}-q^{n+1}}(-\cancel{x^{q+q^{n-r+1}}}+\cancel{x^{{q^{n-r+1}+q^{n+1}}}}+\cancel{x^{1+q^{n-r}}}-\cancel{x^{q^{n-r}+q^{n}}})\\
-x^{q^{2r+1}-q^{n+1}}(y^q-x^{q^{r+1}+q}-\cancel{x^{q^{n-r+1}+q}}+\cancel{x^{q^{n+1}+q^{n-r+1}}})+(y-x^{q^{n-r}+1})^{q^{r+1}}\\
+x^{q^{2r+1}-q^{n+1}}(y-x^{q^r+1}-\cancel{x^{q^{n-r}+1}}+\cancel{x^{q^n+q^{n-r}}})+x^{q^{2r+1}-q^{2r}-q^{n+1}+q^n}(y-x^{q^{n-r}+1})^{q^r}.
\end{split}
\end{equation*}

 After rearranging terms, we have
\begin{equation*}
\begin{split}
 w_2^{q^n}=x^{q^{2r+1}-q^{n+1}}(w_1+y-y^q-x^{q^r+1}-x^{q^{n-r}+1})\\
+(y-x^{q^{n-r}+1})^{q^{r+1}}+x^{q^{2r+1}-q^{2r}-q^{n+1}+q^n}(y-x^{q^{n-r}+1})^{q^r}\\
+x^{q^{2r+1}-q^{n+1}+q^{r+1}+q}.
\end{split}
\end{equation*}  
 It can be checked that all terms of such sum have different  order at $Q$. The triangle inequality
 gives  $v_Q( w_2^{q^n})=v_Q( x^{q^{2r+1}-q^{n+1}+q^{r+1}+q})$, and then $v_Q( w_2)=-(q^{2r}-q^{n}+q^{r}+1)$.
\end{enumerate}           
         
\end{proof}

We are ready to prove the main result of this section.

\begin{proof}[Proof of Theorem \ref{the:weiertrass_gen}]
 The inclusion 
 $$
 H(Q) \supseteq \left\langle q^{n-1},q^{n-1}+q^{r-1},q^n+q^{n-r},q^{2r-1}+q^{n-r-1}, q^{2r}-q^n+q^r+1\right\rangle
 $$
 follows immediately from the previous lemmas.
 

To prove equality, we first show that the semigroup $S=\langle q^{n-1},q^{n-1}+q^{r-1},q^{2r-1}+q^{n-r-1},q^{2r}-q^n+q^r+1\rangle$ is telescopic, see \cite[Definition 6.1]{Kirfel_Pellikaan_telescopic}. Using the notation therein we find that
\begin{itemize}
 \item $d_1=q^{n-1}$ and $S_1=\langle1\rangle$;
 \item $d_2=q^{r-1}$ and $S_2=\left\langle q^{n-r},q^{n-r}+1\right\rangle$;
 \item $d_3=q^{n-r}$ and $S_3=\left\langle q^{r-1},q^{r-1}+q^{2r-n-1},q^{r}+1\right\rangle$; and
 \item $d_4=q^{n-r-1}$ and  $S_4=\left\langle q^{r},q^{r}+q^{2r-n},q^{r+1}+q,q^{3r-n}+1\right\rangle$.
  \item $d_5=1$ and  $S_5= S$.
\end{itemize}
Clearly, $q^{n-r}+1\in S_1$ and  $q^{r}+1\in S_2$. Also,
$$
q^{3r-n}+1=(q^{2r-n+1}-q)q^{r-1}+(q^{r}+1)\in S_3,
$$
and
$$
q^{2r}-q^n+q^r+1=(q^r-q^{n-r}-q^{2r-n}+1)q^{r}+(q^{3r-n}+1)\in S_4.
$$
Therefore, $S$ is telescopic. In particular it is symmetric (see \cite[Lemma 6.5]{Kirfel_Pellikaan_telescopic}). To finish the proof of the theorem, all we are left to show is that the genus of $S$ is equal to the genus of the curve given by \eqref{eq:curve_many_pnts}. 
To compute the genus of $S$ we use the formula (see \cite[Lemma 6.5]{Kirfel_Pellikaan_telescopic})
$$
g(S)=\left(\sum_{i=1}^5 (d_{i-1}/d_i-1)a_i+1\right)/2,
$$
for the genus of a telescopic semigroup. This shows that $g(S)=q^r(q^{n-1}+1)/2$, which is indeed the genus of our curve (see Theorem \ref{the:herm_gen}).
\end{proof}




\section{Acknowledgments}

The first author was partially supported by FAPESP-Brazil grant 2011/19446- 3.

\bibliography{borges_min}
\bibliographystyle{amsalpha}

\end{document}